\documentclass[12pt]{article}

\usepackage{graphicx}
\usepackage[dvips]{epsfig}
\usepackage{amsmath}
\usepackage{amsthm}
\usepackage{amssymb}
\usepackage[mathscr]{eucal}
\usepackage[all]{xy}
\usepackage{exscale,relsize}
\usepackage{psfrag}
\usepackage{bbm}

\title{\bf \Large Integral HOMFLY-PT and $sl(n)$-link homology}
\author{Daniel Krasner}
\date{}

\theoremstyle{plain}
\newtheorem{theorem}{Theorem}

\newtheorem{lemma}[theorem]{Lemma}

\newtheorem{claim}[theorem]{Claim}

\newtheorem{question}[theorem]{Question}

\newtheorem{defn}[theorem]{Definition}

\def\title{\em}

\usepackage[top=1in, bottom=1in, left=1in, right=1in]{geometry}

\renewcommand{\xi}{x_i}

\newcommand{\TenR}{\otimes}

\newcommand{\ii}{\underline{\textbf{\textit{i}}}}
\newcommand{\jj}{\underline{\textbf{\textit{j}}}}

\newcommand{\Hom}{{\rm Hom}}

\newcommand{\HOM}{{\rm HOM}}

\newcommand{\define}{\stackrel{\mbox{\scriptsize{def}}}{=}}

\newcommand{\mc}[1]{\mathcal{#1}}

\newcommand{\ig}[2]{\vcenter{\xy (0,0)*{\includegraphics[scale=#1]{#2}} \endxy}}

\newcommand{\igc}[2]{\begin{center} \includegraphics[scale=#1]{#2} \end{center}}

\newcommand{\auptob}[2]{\xy  (0,-5)*+{#1}="1"; (0,5)*+{#2}="2"; {\ar@{|->} "1";"2"};\endxy}

\newcommand{\olC}{\overline{C}}
\newcommand{\olH}{\overline{H}}
\newcommand{\olR}{\overline{R}}

\def\R{{\mathbb R}}
\def\Z{{\mathbb Z}}
\def\Q{{\mathbb Q}}

\def\1{{\mathbbm 1}}

\begin{document}

\baselineskip 14pt
\maketitle
 
\vspace{0.1in} 

\begin{abstract} Using the diagrammatic calculus for Soergel bimodules, developed by B. Elias and M. Khovanov, as well as Rasmussen's spectral sequence, we construct an integral version of HOMFLY-PT and $sl(n)$-link homology.
\end{abstract}

\vspace{0.1in}

\footnotetext[1]{The author was partially supported by NSF grants DMS-0739392 and DMS-0706924.}

\tableofcontents

\pagebreak

\section{Introduction}
\label{sec-introduction}

\hspace{4mm} During the past half-decade categorification and, in particular, that of topological invariants has flourished into a subject of its own right. It has been a study finding connections and ramifications over a vast spectrum of mathematics, including areas such as low-dimensional topology, representation theory, algebraic geometry, as well as others. Following the original work of M. Khovanov on the categorification of the Jones polynomial in \cite{Kh1}, came a spew of link homology theories lifting other quantum invariants. With a construction that utilized a tool previously developed in an algebra-geometric context, matrix factorizations, M. Khovanov and L. Rozansky produced the $sl(n)$ and HOMFLY-PT link homology theories. Albeit computationally intensive, it was clear from the onset that thick interlacing structure was hidden within. The most insightful and influential work in uncovering these innerconnections was that of J. Rasmussen in \cite{Ras2}, where he constructed a spectral sequence from the HOMFLY-PT to the $sl(n)$-link homology. This was a major step in deconstructing the pallet of how these theories come together, yet many structural questions remained and still remain unanswered, waiting for a new approach. Close to the time of the original work, M. Khovanov produced an equivalent categorification of the HOMFLY-PT polynomial in \cite{K1}, but this time using Hochschild homology of Soergel bimodules and Rouquier complexes of \cite{Rou1}. The latter proved to be more computation-friendly and was used by B. Webster to calculate many examples in \cite{W2}.

In the meantime, a new flavor of categorification came into light. With the work of A. Lauda and M. Khovanov on the categorification of quantum groups in \cite{KL}, a diagrammatic calculus originating in the study of $2$-categories arrived into the foreground. This graphical approach proved quite fruitful and was soon used by B. Elias and M. Khovanov to rewrite the work of Soergel in \cite{EKh}, and en suite by B. Elias and the author to repackage Rouquier's complexes and to prove that they are functorial over braid-cobordisms \cite{EK} (not just projectively functorial as was known before). An immediate advantage to this construction was the inherent ease of calculation, at least comparative ease, and the fact that it worked equally well over the integers as well as over fields. 

As there has yet to be seen an integral version of either HOMFLY-PT or $sl(n)$-link homology, with the original Khovanov homology being defined over $\Z$ and torsion playing an interesting role, a natural question arose as to whether this graphical calculus could be used to define these. The definition of such integral theories is precicely the purpose of this paper. The one immediate disadvantage to the graphical approach is that at the present moment there does not exist a diagrammatic calculus for the Hochschild homology of Soergel bimodules. Hence, to define integral HOMFLY-PT homology, our paper takes a rather roundabout way, jumping between matrix factorizations and diagrammatic Rouquier complexes whenever one is deemed more advantageous than the other. For the $sl(n)$ version of the story, we add the Rasmussen spectral sequence into the mix and essentially repeat his construction in our context. 

The organization of the paper is the following: in section \ref{sec-toolkit} we give a brief account of the necessary tools (matrix factorizations, Soergel bimodules, Hochschild homology, Rouquier complexes, and corresponding diagrammatics) - the emphasis here is brevity and we refer the reader to more original sources for particulars and details; in sections \ref{sec-complex} and \ref{sec-rmoves} we describe the intergal HOMFLY-PT complex and prove the Reidemeister moves, utilizing all of the background in \ref{sec-toolkit}; section \ref{sec-ss} is devoted to the Rasmussen spectral sequence and integral $sl(n)$-link homology, and we conclude it with some remarks and questions.

Throughout the paper we will refer to a positive crossing as the one labelled $D_+$ and negative crossing as the one labelled $D_-$ in figure \ref{resolutions}. For resolutions of a crossing we will refer to $D_o$ and $D_s$ of figure \ref{resolutions} as the ``oriented" and ``singular" resolutions, respectively. We will use the following conventions for the HOMLFY-PT polynomial 

$$aP(D_-) - a^{-1}P(D_+) = (q-q^{-1})P(D_o),$$

\hspace{-8mm} with P of the unknot being $1$. Substuting $a=q^n$ we arrive at the quantum $sl(n)$-link polynomial.

\begin{figure} [!htbp]
\centerline{
\includegraphics[scale=.7]{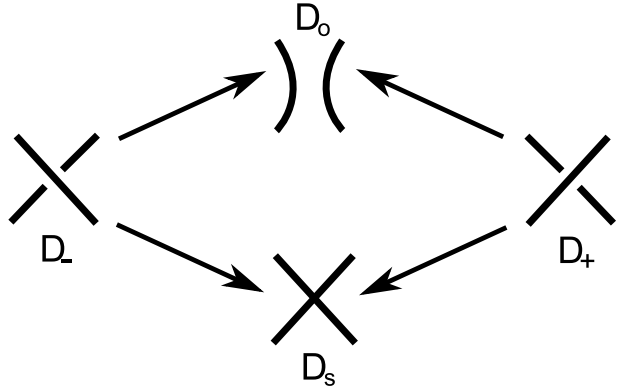}}
\caption{Crossings and resolutions} \label{resolutions}
\end{figure}

\textbf{Acknowledgments:} I would like to thank my advisor M. Khovanov for all his time and help in exploring the subjects at hand over the past few years. In addition, I thank B. Elias and P. Vaz for helpful conversations and e-mail exchanges.

\section{The toolkit}
\label{sec-toolkit}

We will require some knowledge of matrix factorizations, Soergel bimodules and Rouquier complexes, as well as the corresponding diagrammatic calculus of Elias and Khovanov \cite{EKh}. In this section the reader will find a brief survery of the necessary tools, and for more details we refer him to the following papers: for matrix factorizations \cite{KR}, \cite{Ras2}, for Soergel bimodules and Rouquier complexes and diagrammatics \cite{EKh}, \cite{EK}, \cite{K1}, \cite{Rou1}, and for Hochschild homology \cite{Ka}, \cite{K1}.

%
\subsection{Matrix factorizations}
\label{subsec-mf}
%

\begin{defn} Let $R$ be a Noetherian commutative ring, $w \in R$, and $C^*$, $*\in \Z$, a free graded $R$-module. A \emph{$\Z$-graded matrix factorization with potential $w$} consists of $C^*$ and a pair of differentials $d_{\pm}:C^* \rightarrow C^{*\pm 1}$, such that $(d_+ + d_-)^2 = w Id_{C^*}$.
\end{defn}

A morphism of two matrix factorizations $C^*$ and $D^*$ is a homomorphism of graded $R$-modules $f:C^* \rightarrow D^*$ that commutes with both $d_+$ and $d_-$. The tensor product $C^* \otimes D^*$ is taken as the regular tensor product of complexes, and is itself a matrix factorization with diffentials $d_+$ and $d_-$. A useful and easy exercise is the following:

\begin{lemma} Given two matrix factorizations $C^*$ and $D^*$ with potenials $w_c$ and $w_d$, respectively, the tensor product $C^* \otimes D^*$ is a matrix factorization with potential $w_c + w_d$.
\end{lemma}

\begin{remark} 
Following Rasmussen \cite{Ras2}, we work with $\Z$-graded, rather than $\Z / 2\Z$-graded, matrix factorizations as in \cite{KR}. The $\Z$-grading implies that  $(d_+ + d_-)^2 = w Id_{C^*}$ is equivalent to 
$$d_+^2 = d_-^2 =0$$ 
and 
$$d_+d_- + d_-d_+ = w Id_{C^*}.$$ 
In the case that $w=0$, we acquire a new $\Z / 2\Z$-graded chain complex structure with differential $d_+ + d_-$.  Supressing the underlying ring $R$ and potential $w$, we will denote the category of graded matrix factorizations by $\mathit{mf}$.
\end{remark}

We also need the notion of complexes of matrix factorizations. If we visualize a collection of matrix factorizations as sitting horizontally in the plane at each integer level, with differentials $d_+$ and $d_-$ running right and left, respectively, we can think of morphisms $\{d_v\}$ between  these as running in the vertical direction. If $d_v^2 = 0$ we get a complex, i.e. all together we have that 
$$d_{\pm}:C^{i,j} \rightarrow C^{i \pm 1, j}, \ \ \ \  d_{v}:C^{i,j} \rightarrow C^{i, j + 1},$$
where we think of $i$ as the \emph{horizontal} grading and $j$ as the \emph{vertical} grading, and will denote these as $gr_h$ and $gr_v$, respectively. 

In addition we will be taking tensor products of complexes of matrix factorizations (in the obvious way) and, just to add to the confusion we will also have homotopies of these complexes as well homotopies of matrix factorizations themselves. These notions will land us in different categories to which we now give some notation.

\begin{itemize}
\item $\mathit{hmf}$ will denote the \emph{homotopy category of matrix factorizations}
\item $\mc{KOM}(\mathit{mf})$ the \emph{category of complexes of matrix factorizations}
\item $\mc{KOM}_h(\mathit{mf})$ \emph{homotopy category of complexes of matrix factorizations}
\item $\mc{KOM}_h(\mathit{hmf})$ the obvious conglomerate.
\end{itemize}

%
\subsection{Diagrammatics of Soergel bimodules}
\label{subsec-sb}
%

The category of Soergel bimodules $\mc{SC}_1$ is generated monoidally over $R$ by objects $B_i$,
$i \in I$, which satisfy

\begin{equation} \label{decomp1} B_i \TenR B_i \cong B_i\{1\} \oplus B_i\{-1\}  \end{equation}
\begin{equation} B_i \TenR B_j \cong B_j \TenR B_i \label{dc-ij} \textrm{ for distant } i, j \end{equation}
\begin{equation} B_i \TenR B_j \TenR B_i \oplus B_j \cong B_j \TenR B_i \TenR B_j \oplus
  B_i \label{dc-ipi} \textrm{ for adjacent } i, j. \end{equation}

(Technically speaking this should be called the category of Bott-Samuelson bimodules and the ``real" category of Soergel bimodules is gotten as described at the end of this section. See also \cite{EKh} and \cite{EK} for more details.)
The Grothendieck group of $\mc{SC}(I)$ is isomorphic to the Hecke algebra $\mc{H}$ of type  $A_\infty$ over the ring $\Z[t, t^{-1}]$, with the class of $B_i$ being
sent to a generator $b_i$ of $\mc{H}$, and the class of $R\{1\}$ being sent to $t$. 

More concretely, the Soergel bimodule $B_i = R \otimes_{R^i} R\{-1\}$, where $R$ is a graded polynomial ring, $\{m\}$ denotes the grading shift by $m$, and $R^i$ is the subring of invariants corresponding to the permutation $(i,i+1)$ under the natural action of $S_n$ on the variables. There is some flexibility as to the exact description of $R$, but in our case it will mainly be the ring $\Z[x_1-x_2, \dots , x_{n-1} - x_n]$ with $\deg{x_i} = 2$ (note that our grading shift of $-1$ in the definition of $B_i$ is absent from the contruction of \cite{K1}). We have that $B_{\emptyset} = R$ itself, and  $B_{\ii} = 
B_{i_1} \otimes B_{i_2} \otimes \dots \otimes B_{i_d}$ where $\ii$ is denotes the sequence $\{i_1, i_2, \dots , i_d\}$, i.e. 

$$B_{\ii} = (R \otimes_{R^{i_1}} R\{-1\}) \otimes (R \otimes_{R^{i_2}} R\{-1\}) \otimes \dots \otimes (R \otimes_{R^{i_d}} R\{-1\})$$
$$= R \otimes_{R^{i_1}} R \otimes R \otimes_{R^{i_2}} R \otimes \dots \otimes R \otimes_{R^{i_d}} R\{-d\}.$$

One useful feature of this categorification is that it is easy to calculate the dimension of Hom spaces in each degree. Let $\HOM(M,N) \define
\bigoplus_{m \in \Z} \Hom(M,N\{m\})$ be the graded vector space (actually an $R$-bimodule) generated by homogeneous morphisms of all degrees. Let
$B_{\ii} \define B_{i_1} \TenR \cdots \TenR B_{i_d}$. Then $\HOM(B_{\ii},B_{\jj})$ is a free left $R$-module, and its graded rank over $R$ is
given by $(b_{\ii},b_{\jj})$. For more information on this categorification and related topics we refer the reader to \cite{EKh}, and \cite{Soe1}.

The graphical counterpart, which we will also refer to as $\mc{SC}_1$ was given a diagrammatic presentation by generators and
relations, allowing morphisms to be viewed as isotopy classes of certain graphs.

An object in $\mc{SC}_1$ is given by a sequence of indices $\ii$, which is visualized as $d$ points
on the real line $\R$, labelled or ``colored'' by the indices in order from left to right. Sometimes
these objects are also called $B_{\ii}$.  Morphisms are given by pictures embedded in the strip $\R
\times [0,1]$ (modulo certain relations), constructed by gluing the following generators
horizontally and vertically:

\igc{1}{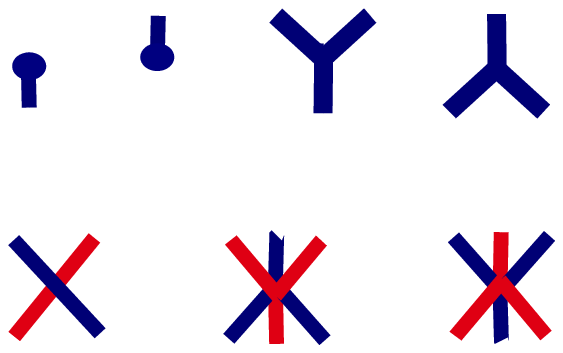}

For instance, if ``blue" corresponds to the index $i$ and ``red" to $j$, then the lower right generator is a morphism from $jij$ to $iji$. The generating
pictures above may exist in various colors, although there are some restrictions based on adjacency conditions.

We can view a morphism as an embedding of a planar graph, satisfying the following properties:
\begin{enumerate}
\item Edges of the graph are colored by indices from $1$ to $n$.
\item Edges may run into the boundary $\R \times \{0,1\}$, yielding two sequences of colored points
  on $\R$, the top boundary $\ii$ and the bottom boundary $\jj$.  In this case, the graph is viewed
  as a morphism from $\jj$ to $\ii$.
\item Only four types of vertices exist in this graph: univalent vertices or ``dots'', trivalent
  vertices with all three adjoining edges of the same color, 4-valent vertices whose adjoining edges
  alternate in colors between $i$ and $j$ distant, and 6-valent vertices whose adjoining
  edges alternate between $i$ and $j$ adjacent.
\end{enumerate}

The degree of a graph is +1 for each dot and -1 for each trivalent vertex. $4$-valent and $6$-valent vertices are of degree $0$. The term \emph{graph} henceforth refers to such a graph embedding.

By convention, we color the edges with different colors, but do not specify which colors match up with which $i \in I$. This is legitimate, as
only the various adjacency relations between colors are relevant for any relations or calculations. We will specify adjacency for all pictures,
although one can generally deduce it from the fact that 6-valent vertices only join adjacent colors, and 4-valent vertices join only distant colors.

In addition to the bimodules $B_{\ii}$ above, we will require the use of the bimodule $R \otimes_{R^{i, i+1}} R\{-3\}$, where $R^{i, i+1}$ is the ring of invariants under the transpositions $(i, i+1)$ and $(i+1, i+2)$, and will use a black squiggly line, as in equation \ref{decomp2d} below, to represent it. This bimodule comes into play in the isomorphisms
 
\begin{equation} \label{decomp2a}  B_i \otimes B_{i+1} \otimes B_i \cong B_i \oplus (R \otimes_{R^{i, i+1}} R\{-3\}) \end{equation}
and 
 
\begin{equation} \label{decomp2b}  B_{i+1} \otimes B_{i} \otimes B_{i+1} \cong B_{i+1} \oplus (R \otimes_{R^{i, i+1}} R\{-3\}), \end{equation}

\hspace{-8mm} which we will use in the proof of Reidemeister move III.
As usual in a diagrammatic category, composition of morphisms is given by vertical concatenation,
and the monoidal structure is given by horizontal concatenation.

We then allow $\Z$-linear sums of graphs, and apply relations  to obtain our category $\mc{SC}_1$. 
The reations come in three flavors: one color, two distant colors, two adjacent and one distant, and three mutually distant colors. We do not list all of them here, just the consequences necessary for the calculations at hand, and refer the reader to \cite{EKh} and \cite{KR} for a complete picture. Our graphs are invariant under isotopy and in addition we have the following isomorphisms or ``decompositions":

\begin{equation} \label{decomp1d}  \ig{.9}{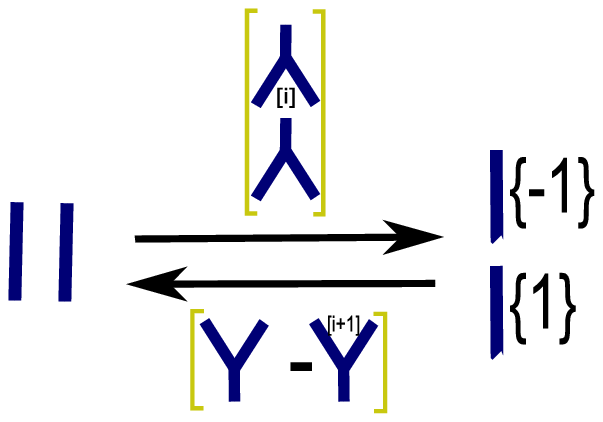}
 \end{equation}

Note that this relation is precisely that of \ref{decomp1} described diagrammatically.

\begin{equation} \label{decomp2d} \ig{.9}{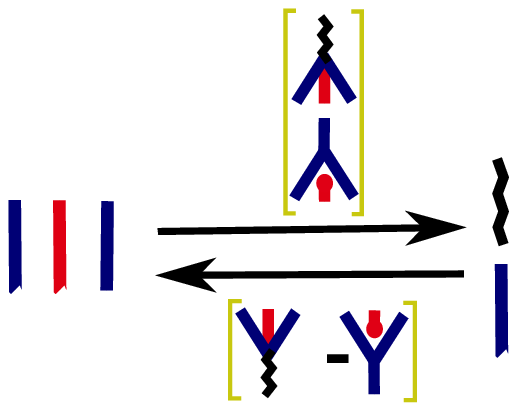} \end{equation}

Here we have the graphical counterpart of \ref{decomp2a} and \ref{decomp2b}.

\begin{remark} There is a functor from this graphical category to the category of $R$-bimodules, sending a
  line colored $i$ to $B_i$ and each generator to an appropriate bimodule map. The functor gives an
  equivalence of categories between this diagrammatic category and the subcategory
  $\mc{SC}_1$ of $R$-bimodules mentioned in the previous section, so the use of the same name is legitimate.

  Our diagrammatic category has many wonderful properties, such as the self-adjointness of $B_i$, which permits us to ``twist" morphisms around and view any morphism as one from or to the empty diagram. This allows for a very hands-on, explicit, understanding of hom-paces between objects in  $\mc{SC}_1$, which was key in proving functoriality in \cite{EK}. 
\end{remark}

Primarily we will work in another category denoted $\mc{SC}_2$, the category formally containing all direct sums and grading shifts of objects in $\mc{SC}_1$, but whose morphisms are forced to be degree 0. In addition, we let $\mc{SC}$ be the Karoubi envelope, or idempotent completion, of the category $\mc{SC}_2$. Recall that the Karoubi envelope of a category $\mc{C}$ has as objects pairs $(B,e)$ where $B$ is an object in $\mc{C}$ and $e$ an idempotent endomorphism of $B$.  This object acts as
though it were the ``image'' of this projection $e$, and in an additive category behaves like a
direct summand. For more information on Karoubi envelopes, see Wikipedia. It is really here that the object  $R \otimes_{R^{i, i+1}} R\{-3\}$ of \ref{decomp2a} and \ref{decomp2b} resides. In practice all our calculations will be done in  $\mc{SC}_2$, but since this includes fully faithfully into $\mc{SC}$ they will be valid there as well.

%
\subsection{Hochschild (co)homology}
\label{subsec-HH}
%

Let $A$ be a $\mathbbm{k}$ algebra and $M$ an $A$-$A$-bimodule, or equivalently a left  $A \otimes A^{op}$-module or a right $A^{op} \otimes A$-module. The definitions of the \emph{Hochschild (co)homology} groups $HH_*(A,M)$ ($HH^*(A,M)$) are the following:

\begin{equation}
HH_*(A,M) := Tor_*^{A \otimes A^{op}}(M,A) \ \ \ \ \ HH^*(A,M) := Ext^*_{A \otimes A^{op}}(A,M).
\label{HHdef}
\end{equation}

To compute this we take a projective resolution of the $A$-bimodule $A$, with the natural left and right action, by projective $A$-bimodules

$$ \dots \rightarrow P_2 \rightarrow P_1 \rightarrow P_0 \rightarrow 0,$$ 

\hspace{-8mm} and tensor this with $M$ over $A \otimes A^{op}$ to get

 $$ \dots \rightarrow P_2 \otimes_{A \otimes A^{op}} M \rightarrow P_1 \otimes_{A \otimes A^{op}} M \rightarrow P_0 \otimes_{A \otimes A^{op}} M \rightarrow 0.$$

The homology of this complex is isomorphic to $HH_*(A,M)$. 

\hspace{-6mm}\textbf{Example:} For any bimodule $M$, we have
$$HH_0(A,M) \cong M/[A,M] \ \ \ \ \ HH^0(A,M) \cong M^A,$$
where $[A,M]$ is the subspace of $M$ generated by all elements of the form $am-ma$,  $a \in A$ and $m \in M$, and $M^A = \{m \in M \ | \ am=ma$ for all a $\in A\}$. We leave this as an exercise or refer the reader to \cite{Ka}.
\\

If we take the polynomial algebra $A = \mathbbm{k}[x_1, \dots , x_n]$, with $\mathbbm{k}$ commutative, then we can use a much smaller, ``Koszul," resolution of $A$ by free  $A\otimes A$-modules. This is gotten by taking the tensor product of the following complexes 

\begin{displaymath}
\xymatrix{
0 \ar[r] & A \otimes A  \ar[rr]^{x_i \otimes 1 - 1 \otimes x_i}  & & A \otimes A  \ar[r] & 0,}   
\end{displaymath}

\hspace{-6mm}for $1 \geq i \geq n$. This resolution has length $n$, and its total space is naturally isomorphic to the exterior algebra on $n$ generators tensored with $A \otimes A$. Hence, we get that the Hochschild homology of a bimodule $M$ over $A$ is made up of $2^n$ copies of $M$, with the differentials coming from multiplication by $x_i \otimes 1 - 1 \otimes x_i$, i.e.

$$ 0 \rightarrow C_n(M) \rightarrow \dots \rightarrow C_1(M) \rightarrow C_0(M) \rightarrow 0,$$

with 

\begin{displaymath}
C_j(M) = \bigoplus_{I \subset \{1, \dots, n \}, |I|=j} M\otimes_{\Z}\Z[I],
\end{displaymath}

\hspace{-6mm}where $\Z[I]$ is the rank $1$ free abelian group generated by the symbol $[I]$ (i.e. it's there to keep track where exactly we are in the complex). Here, the differential takes the form 

$$d(m\otimes[I]) = \sum_{i\in I} \pm (x_im-mx_i) \otimes [I \backslash \{i\}],$$

\hspace{-6mm}and the sign is taken as negative if $I$ contains an odd number of elements less than $i$. 

\begin{remark}For the polynomial algebra, the Hochschild homology and cohomology are isomorphic,

$$HH_i(A,M) \cong HH^{n-i}(A,M),$$

\hspace{-7mm} for any bimodule $M$. This comes from self-duality of the Koszul resolution for such algebras. Hence, we will be free to use either homology or cohomology groups in the constructions below. 
\end{remark}
\\

For us, taking Hochschild homology will come into play when looking at closed braid diagrams. To a given resolution of a braid diagram we will assign a Soergel bimodule;  ``closing off" this diagram will correspond to taking Hochschild homology of the associated bimodule. More details of this below in section \ref{subsec-sbconstruction}.

\section{The integral HOMFLY-PT complex}
\label{sec-complex}

%
\subsection{The matrix factorization construction}
\label{subsec-mfconstruction}
%

As stated above we will work with $\Z$-graded, rather than ${\Z} / {2\Z}$-graded, matrix factorizations and follow closely the conventions laid out in \cite{Ras2}. We begin by first assigning the appropriate complex to a single crossing and then extend this to general braids.\\

\textbf{Gradings:} 
Our complex will be triply graded, coming from the internal or ``quantum" grading of the underlying ring, the homological grading of the matrix factorizations, and finally an overall homological grading of the entire complex. It will be convenient to visualize our complexes in the plane with the latter two homological gradings lying in the horizontal and vertical directions, respectively. We will denoted these  gradings by $(i, j, k) = ( q, 2gr_h, 2gr_v)$ and their shifts by curly brackets, i.e. $\{a, b, c\}$ will indicate a shift in the quantum grading by $a$, in the horizontal grading by $b$, and in the vertical grading by $c$. Note that following the conventions in  \cite{Ras2} we have doubled the latter two gradings.   

\begin{figure} [!htbp]
\centerline{
\includegraphics[scale=.7]{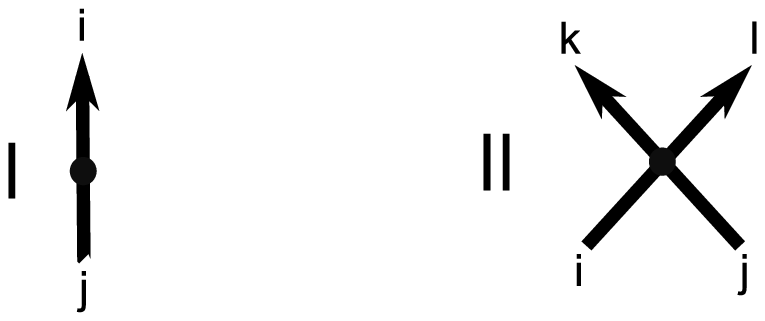}}
\caption{} \label{vertrel}
\end{figure}

\begin{defn}\textbf{\{Edge ring\}} Given a diagram $D$ with vertices labelled by $x_1, \dots , x_n$, define the \emph{edge ring} of D as $R(D) := \Z [x_1, \dots , x_n]/<rel(v_i)>$, where $i$ runs over all internal vertices, or marks, with the defining relations being $x_i - x_j$ for type I and $x_k + x_l - x_i - x_j$ for type II vertices (see figure \ref{vertrel}).
\end{defn}
 Consider the two types of crossings $D_+$ and $D_-$, as in figure \ref{resolutions}, with outgoing edges labeled by $k, l$, and incoming edges labelled by $i, j$ . Let
$$R_c := \Z[x_i, x_j, x_k, x_l]/(x_k + x_l - x_i - x_j) \cong \Z[x_i, x_j, x_k]$$

\hspace{-8mm} be the underlying ring associated to a crossing. To the positive crossing $D_+$ we assign the following complex:

\begin{displaymath}
\xymatrix{
R_c \{0, -2, 0\}  \ar[rrr]^{(x_k - x_i)}    & & & R_c \{0, 0, 0\}   \\
R_c \{2, -2, -2\}   \ar[rrr]^{-(x_k - x_i)(x_k - x_j)}  \ar[u]^{(x_j - x_k)}  & & & R_c \{0, 0, -2\} \ar[u]^{1} }
\end{displaymath}

To the negative crossing $D_-$ we assign the following complex:

\begin{displaymath}
\xymatrix{
R_c \{0, -2, 2\}  \ar[rrr]^{-(x_k - x_i)(x_k - x_j)}    & & & R_c \{-2, 0, 2\}   \\
R_c \{0, -2, 0\}   \ar[rrr]^{(x_k - x_i)}  \ar[u]^{1}  & & & R_c \{0, 0, 0\} \ar[u]^{(x_j - x_k)} }
\end{displaymath}

\textbf{A few useful things to note:} The horizontal and vertical differentials $d_+$ and $d_v$ are homogeneous of degrees $(2,2,0)$ and $(0,0,2)$, respectively. For those more familiar with \cite{KR} and hoping to reconcile the differences, note that in $R_c$ multiplication by  $x_k x_l- x_i x_j = -(x_k - x_i)(x_k - x_j)$, so up to some grading shifts we are really working with the same underlying complex as in the original construction, but of course now over $\Z$, not $\Q$.

To write down the complex for a general braid we tensor the above for every crossing, keeping track of markings, and then replace the underlying ring with a copy of the edge ring $R(D)$. More precisely, given a diagram $D$ of a braid let
$$ C(D) := \bigotimes_{crossings} (C(D_c) \otimes_{R_c} R(D)).$$

\begin{defn} \textbf{\{HOMFLY-PT homology \}}Given a braid diagram $D$ of a link $L$ we define its HOMFLY-PT homology to be the group
$$H(L):= H(H(C(D), d_+), d_v ^*)\{-w+b, w+b-1, w-b+1\},$$
where $w$ and $b$ are the writhe and the number of strands of $D$, respectively. 
\end{defn}

\begin{remark} In \cite{Ras2}, this is what J. Rasmussen calls the ``middle HOMFLY homology." 
The relation between this link homology theory and the HOMFLY-PT polynomial is that for any link $L \subset S^3$ 
$$\sum_{i,j,k} (-1)^{(k-j)/2} a^j q^i dim H^{i,j,k}(L) = \frac{-P(L)}{q-q^{-1}}.$$
\end{remark}\\

\textbf{The reduced complex:} There is a natural subcomplex $\olC(D) \subset C(D)$ defined as follows: let $\olR (D) \subset R(D)$ to be the subring generated by $x_i - x_j$ where $i, j$ run over all edges of $D$ and let $\olC(D)$ be the subcomplex gotten by replacing in $C(D)$ each copy of $R(D)$ by one of  $\olR (D)$. A quick glance at the complexes $C(D_+)$ and $C(D_-)$ will reassure the reader that this is indeed a subcomplex, as the coefficients of both $d_v$ and $d_+$ lie in $\olR(D)$. We will refer to $\olC(D)$ as the \emph{reduced complex} for D.

\begin{itemize}
\item  If $i$ is an edge of $D$ we can also define the complex $\olC(D,i) := C(D)/(x_i)$. It is not hard to see that
$\olC(D,i) \cong \olC(D)$ and is, hence, independent of the choice of edge $i$. See \cite{Ras2} section 2.8 for a discussion as well as \cite{KR}. 
\end{itemize}

Below we will work primarily with the reduced complex $\olC(D)$, and will stick with the grading conventions of \cite{Ras2}, which are different than that of \cite{KR}. 

\begin{defn} \textbf{\{reduced homology\}}Given a braid diagram $D$ of a link $L$ we define its reduced HOMFLY-PT homology to be the group $$\olH(L):= H(H(\olC(D), d_+), d_v ^*)\{-w+b-1, w+b-1, w-b+1\},$$
where $w$ and $b$ are the writhe and the number of strands of $D$, respectively. 
\end{defn}

\begin{remark} For any link $L \subset S^3$ we have 
$$\sum_{i,j,k} (-1)^{(k-j)/2} a^j q^i dim \olH^{i,j,k}(L) = P(L).$$  
We can look at the complex $C(D)$ in two essential ways: either as the tensor product, over appropriate rings, of $C(D_+)$ and $C(D_-)$ for every crossing in our diagram $D$ (as described above), or as a tensor product of corresponding complexes over all resolutions of the diagram. Although this is really just a matter of point of view, the latter approach is what we find in the original construction of Khovanov and Rozansky, as well as in the Soergel bimodule construction to be described below. To clarify this approach, consider the oriented $D_o$ and singular $D_s$ resolution of a crossing as in diagram \ref{resolutions}. Assign to $D_o$ the complex

\begin{displaymath}
\xymatrix{
0 \ar[r] & R_c  \ar[rr]^{(x_k - x_i)}  & & R_c \ar[r] & 0}   
\end{displaymath}  

and to $D_s$ the complex

\begin{displaymath}
\xymatrix{
0 \ar[r] & R_c   \ar[rrr]^{-(x_k - x_i)(x_k - x_j)}   & & & R_c  \ar[r] & 0.}
\end{displaymath} 

Then we have 
$$C(D_+): \ \ 0 \rightarrow C(D_s) \longrightarrow C(D_o) \rightarrow 0,$$
 
$$C(D_-): \ \ 0 \rightarrow C(D_o) \longrightarrow C(D_s) \rightarrow 0,$$
where the maps are given by $d_v$ as defined above.
[For simplicity we leave out the internal grading shifts.] Let a resolution of a link diagram $D$ be a resolution of each crossing in either of the two ways above, and let the complex assigned to each resolution be the tensor product of the corresponding complexes for each resolved crossing. Then, modulo grading shifts, our total complex can be viewed as 
$$C(D) = \bigoplus_{resolutions} C(D_{res}),$$
where $D_{res}$ is the diagram of a given resolution. This closely mimics the ``state-sum model" for the Jones polynomial, due to Kauffman \cite{Kau}, or the MOY calculus of \cite{MOY} for other quantum polynomials. 
\end{remark}

%
\subsection{The Soergel bimodule construction}
\label{subsec-sbconstruction}
%

We now turn to the Soergel bimodule construction for the HOMLFY-PT homology of \cite{K1}. Recall from section \ref{subsec-sb} that the Soergel bimodule $B_i = R \otimes_{R^i} R\{-1\}$ where $R = \Z[x_1 - x_2, \dots , x_{n-1} - x_n]$ is the ring generated by consecutive differences in variables $x_1, \dots, x_n$ ($n$ is the number of strands in the braid diagram), and $R^i \subset R$ is the subring of $S_2$-invariants corresponding to the permutation action $x_i \leftrightarrow x_{i+1}$. Furthermore define the map $B_i \rightarrow R$ by  $1\otimes1 \longmapsto 1,$  and the map $R \rightarrow B_i$ by $1 \longmapsto (x_i - x_{i+1})\otimes 1 +1 \otimes (x_i - x_{i+1})$. We resolve a crossing in position $[i, i+1]$  in the either of the two ways, as in figure \ref{resolutions}, assigning $R$ to the oriented resolution and $B_i$ to the singular resolution. For a positive crossing we have the complex
$$C(D_+): \ \ 0 \rightarrow R \{2\} \longrightarrow B_i \{1\} \rightarrow 0,$$
and for a negative crossing the complex
$$C(D_-): \ \ 0 \rightarrow B_i \{-1\} \longrightarrow R \{-2\} \rightarrow 0.$$
We place $B_i$ in homological grading $0$ and increase/decrease by $1$, i.e. in the complex for $D_+$, $R\{2\}$ is in homological grading $-1$. Note, this grading convention differs from \cite{K1}, and is the convention used in \cite{EK}. The complexes above are known as Rouquier complexes, due to R. Rouquier who studied braid group actions with relation to the category of Soergel bimodules; for more information we refer the reader to \cite{EK}, \cite{K1}, and \cite{Rou1}.

\begin{figure} [!htbp]
\centerline{
\includegraphics[scale=.6]{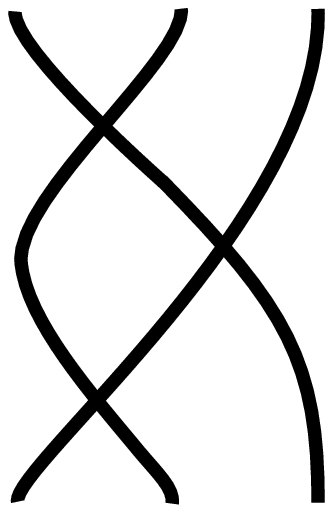}}
\label{exres}
\caption{}
\end{figure}

Given a braid diagram $D$ we tensor the above complexes for each crossing, arriving at a total complex of length $k$, where $k$ is the number of crossings of $D$, or equivalently the length of the corresponding braid word. Each entry in the complex can be thought of as a resolution of the diagram consisting of the tensor product of the appropriate Soergel bimodules. For example, to the graph in \ref{exres} we assign the bimodule $B_1 \otimes B_2 \otimes B_1$. That is,  modulo grading shifts, we can view our total complex as 
$$C(D) = \bigoplus_{resolutions} C(D_{res}).$$
To proceed, we take Hochschild homology $HH(C(D_{res}))$ for each resolution of $D$ and arrive at the complex
$$HH(C(D)) = \bigoplus_{resolutions} HH(C(D_{res})),$$
with the induced differentials. Finally, taking homology of $HH(C(D))$ with respect to these differentials gives us our link homology.

\begin{defn} \textbf{\{reduced homology\}} Given a braid diagram $D$ of a link $L$ we define its reduced HOMFLY-PT homology to be the group $$H(HH(C(D))).$$ 
\end{defn} 

Of course, now that we have defined reduced HOMFLY-PT homology in two different ways, it would be nice to reconcile the fact that they are indeed the same.

\begin{claim}
\label{claim-equivalence}
Up to grading shifts the two definitions of reduced HOMFLY-PT homology agree, i.e. $H(H(\olC(D), d_+), d_v ^*) \cong H(HH(C(D)))$ for a diagram $D$ of a link $L$.  

\end{claim}

\begin{proof}
The proof in \cite{K1} works without any changes for matrix factorizations and Soergel bimodules over $\Z$. We sketch it here for completeness and the fact that we will be referring to some of its details a bit later.  Lets first look at the matrix factorization $C(D_s)$ (unreduced version) associated to a singular resolution $D_s$.  Now $C(D_s)$ can be though of as a Koszul complex of the sequence $(x_k + x_l - x_ i- x_j, x_k x_l - x_i x_j)$ in the polynomial ring $\Z[x_i, x_j, x_k, x_l]$ (don't forget that in $R_c$ multiplication by  $x_k x_l- x_i x_j = -(x_k - x_i)(x_k - x_j)$). Now this sequence is regular and the complex has cohomology in the right-most degree. The cohomology is the quotient ring 
$$\Z[x_i, x_j, x_k, x_l]/(x_i + x_j - x_k - x_l, x_k x_l - x_i x_j).$$
This is naturally isomorphic to the Soergel bimodule $B'_i$ (notice that this is the ``unreduced" Soergel bimodule) over the polynomial ring $\Z[x_i, x_j]$. The left and right action of $R'$ on $B'_i$ corresponds to multiplication by $x_i, x_j$ and $x_k, x_l$, respectively. Quotienning out by $x_k + x_l - x_ i- x_j$ and $x_k x_l - x_i x_j$ agrees with the definition of $B'_i$ as the tensor product $R' \otimes_{R'_i} R'$ over the subalgebra $R'$ of symmetric polynomials in $x_1, x_2$.  

Now lets consider a general resolution $D_{res}$. The matrix factorization for $D_{res}$ is, once again, just a Koszul complex corresponding to a sequence of two types of elements. The first ones are as above, i.e. they are of the form $x_k + x_l - x_ i- x_j$ and $x_k x_l - x_i x_j$ and come from the singular resolutions $D_s$, and the remaining are of the form $x_i - x_j$ that come from ``closing off" our braid diagram $D$, which in turn means equating the corresponding marks at the top and bottom the diagram. Now it is pretty easy to see that the polynomials of the first type, coming from the $D_s$'s form a regular sequence and we can quotient out by them immediately, just like above. The quotient ring we get is naturally isomorphic to the Soergel bimodule $B'(D_{res})$ associated to the resolution $D_{res}$. At this point all we have left is to deal with the remaining elements of the form $x_i - x_j$ coming from closing off $D$; to be more concrete, the Koszul complex we started with for $D_{res}$ is quasi-isomorphic to the Koszul complex of the ring  $B'(D_{res})$ corresponding to these remaining elements. This in turn precisely computes the Hochschild homology of $B'(D_{res})$. 

Finally if we downsize from $B'_i$ to $B_i$ and from $C(D_{res})$ to $\olC(D_{res})$ we get the required isomorphism. For more details we refer the reader to \cite{K1}.
\end{proof}

\textbf{Gradings et all:} We come to the usual rigmarole of grading conventions, which seems to be evepresent in link homology. Perhaps when using the Rouquier complexes above we could have picked conventions that more closely matched those of \ref{subsec-mfconstruction}. However, we chose not to for a couple of reasons: first there would inevitably be some grading conversion to be done either way due to the inherent difference in the nature of the constructions, and second we use Rouquier complexes to aid us in just a few results (namely the proof of Reidemeister moves II and III), and leave them shortly after attaining these; hence, it is convenient for us, as well as for the reader familiar with the Soergel bimodule construction of \cite{K1} and the diagrammatic construction of \cite{EKh}, to adhere to the conventions of the former and the subsequent results in \cite{EK}. For completeness, we descibe the conversion rules. Recall that in the matrix factorization construction of \ref{subsec-mfconstruction} we denoted the gradings as $(i, j, k) = ( q, 2gr_h, 2gr_v)$.

\begin{itemize}
\item To get the cohomological grading in the Soergel construction take $(j-k)/2$ from \ref{subsec-mfconstruction}.
\item The Hochschild here matches the ``horizontal" or $j$ grading of \ref{subsec-mfconstruction}. 
\item To get the ``quantum" grading $i$ of \ref{subsec-mfconstruction} of an element $x$, take Hochschild grading of $x$ minus $\deg(x)$, i.e. $\deg(x) = j(x) - i(x)$. 
\end{itemize}

%
\subsubsection{Diagrammatic Rouquier complexes}
\label{subsec-diagrouquier}
%

We now restate the last section in the diagrammatic landuage of \cite{EK} as outlined above in \ref{subsec-sb}. The main advantage of doing this is the inherent ability of the graphical calculus developed by Elias and Khovanov in \cite{EKh} to hide and, hence simplify, the complexity of the calculations at hand. Recall that we work in the integral version of Soergel category $\mc{SC}_2$ as defined in section 2.3 of \cite{EK}, which allows for constructions over $\Z$ without adjoining inverses (see section 5.2 in \cite{EK} for a discussion of these facts).  Recall, that an object of  $\mc{SC}_2$ is given by a sequence of indices $\ii$, visualized as $d$ points on the real line and morhisms are given by pictures or graphs embedded in the strip  $\R \times [0,1]$. We think of the indices as ``colors," and depict them accordingly. The Soergel bimodule $B_i$ is represented by a vertical line of ``color" $i$ (i.e. by the identity morphism from $B_i$ to itself) and the maps we find in the Rouquier complexes above, section  \ref{subsec-sbconstruction}, are given by those referred to as ``start-dot" and ``end-dot." More precisely, the complexes $C(D_-)$ and $C(D_+)$ become

\begin{figure} [!htbp]
\centerline{
\includegraphics[scale=.6]{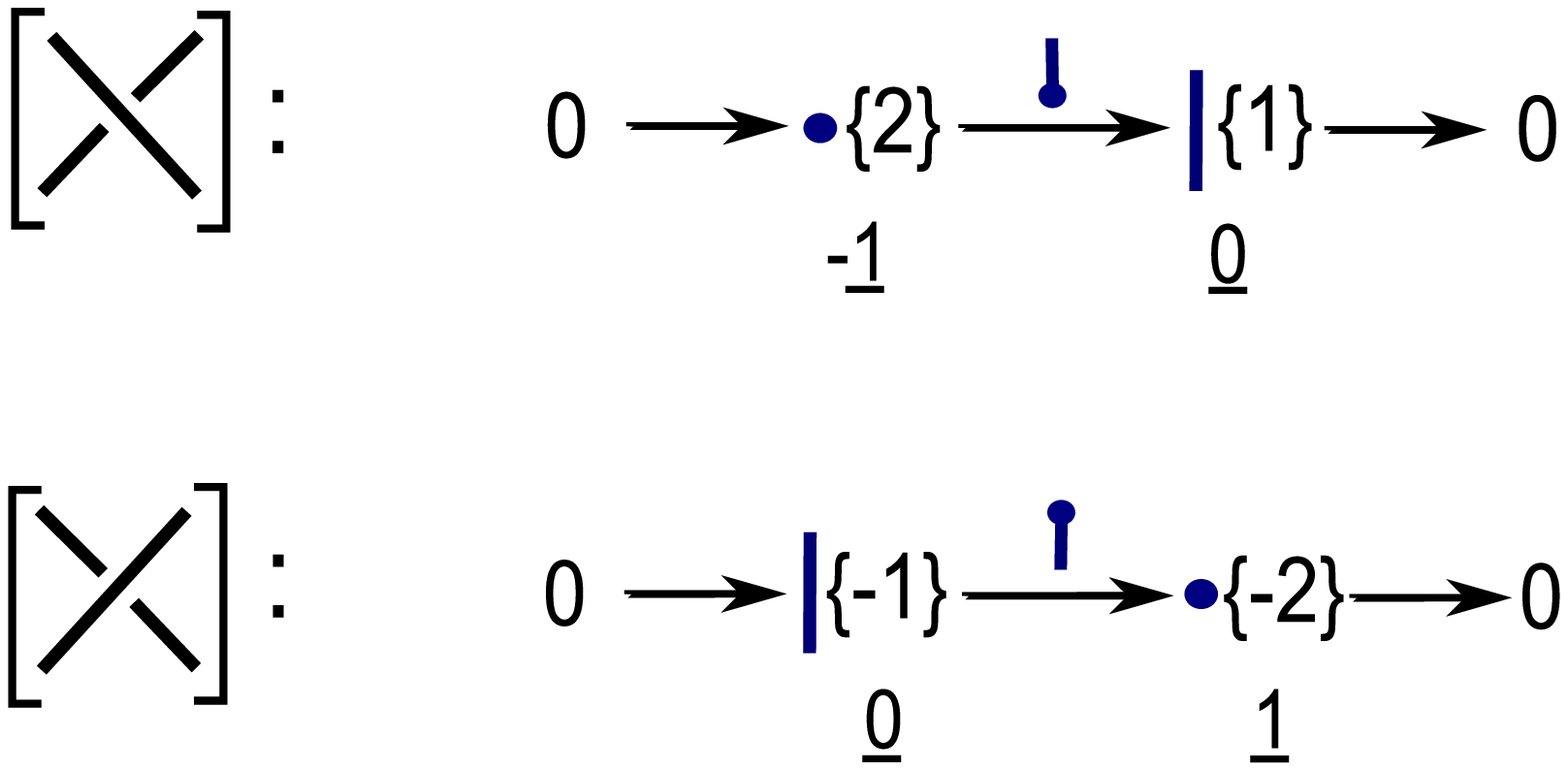}}
\caption{Diagrammatic Rouquier complex for right and left crossings} \label{crossingsdef}
\end{figure}

For completeness and ease we remind the reader of the diagrammatic calculus rubric used to contruct Rouquier complexes for a given braid diagram.

%
\subsubsection{Conventions}
\label{subsec-conventions}
%

We use a colored circle to indicate the empty graph, but maintain the color for reasons of sanity.
It is immediately clear that in the complex associated to a tensor product of $d$ Rouquier complexes, each
summand will be a sequence of $k$ lines where $0 \leq k \leq d$ (interspersed with colored
circles, but these represent the empty graph so could be ignored). Each differential from one
summand to another will be a ``dot'' map, with an appropriate sign.

\begin{enumerate}

 \item The dot would be a map of degree 1 if $B_i$ had not been shifted accordingly. In $\mc{SC}_2$, all maps must be homogeneous, so we could have deduced
the degree shift in $B_i$ from the degree of the differential. Because of this, it is not useful to keep track of various degree shifts of objects in a
complex. Hence at times we will draw all the objects without degree shifts, and all differentials will therefore be maps of graded degree 1 (as well as homological degree
1). It follows from this that homotopies will have degree -1, in order to be degree 0 when the shifts are put back in. One could put in the degree shifts
later, noting that $B_{\emptyset}$ always occurs as a summand in a tensor product exactly once, with degree shift 0.

\item We will use blue for the index associated to the leftmost crossing in the braid, then red and
  dotted orange for other crossings, from left to right. The adjacency of these various colors is
  determined from the braid.

\item We read tensor products in a braid diagram from bottom to top.  That is, in the following
  diagram, we take the complex for the blue crossing, and tensor by the complex for the red
  crossing. Then we translate this into pictures by saying that tensors go from left to right. In
  other words, in the complex associated to this braid, blue always appears to the left of red.

\vspace{3mm}

\begin{figure} [!htbp]
\centerline{
\includegraphics[scale=1.2]{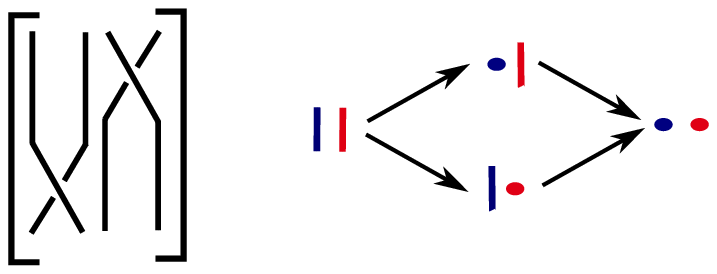}}
\label{example1}
\end{figure}\

\vspace{-5mm}

\item One can deduce the sign of a differential between two summands using the Liebnitz rule,
  $d(ab)=d(a)b + (-1)^{|a|}ad(b)$.  In particular, since a line always occurs in the basic complex
  in homological dimension $\pm 1$, the sign on a particular differential is exactly given by the
  parity of lines appearing to the left of the map.  For example,

\vspace{3mm}

\begin{figure} [!htbp]
\centerline{
\includegraphics[scale=.9]{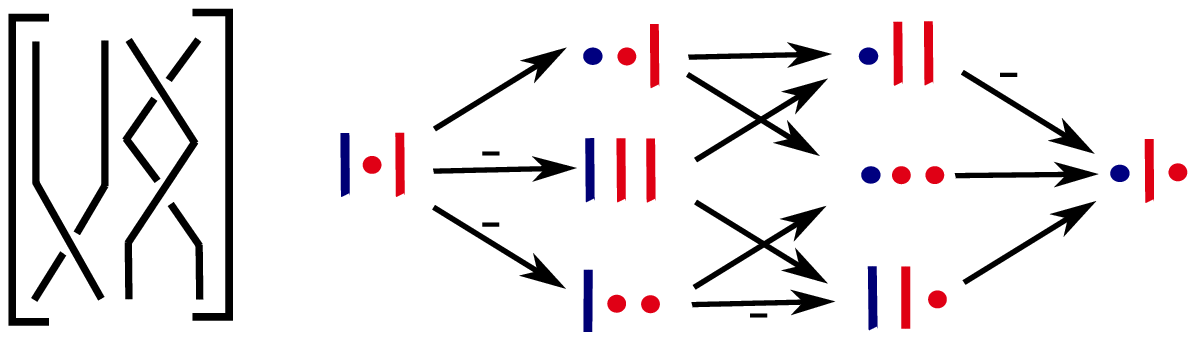}}
\label{example2}
\end{figure}\

\vspace{-5mm}

\item When putting an order on the summands in the tensored complex, we use the following
  standardized order.  Draw the picture for the object of smallest homological degree, which we draw
  with lines and circles.  In the next homological degree, the first summand has the first color
  switched (from line to circle, or circle to line), the second has the second color switched, and
  so forth.  In the next homological degree, two colors will be switched, and we use the
  lexicographic order: 1st and 2nd, then 1st and 3rd, then 1st and 4th... then 2nd and 3rd,
  etc. This pattern continues.

\vspace{3mm}

\begin{figure} [!htbp]
\centerline{
\includegraphics[scale=1.1]{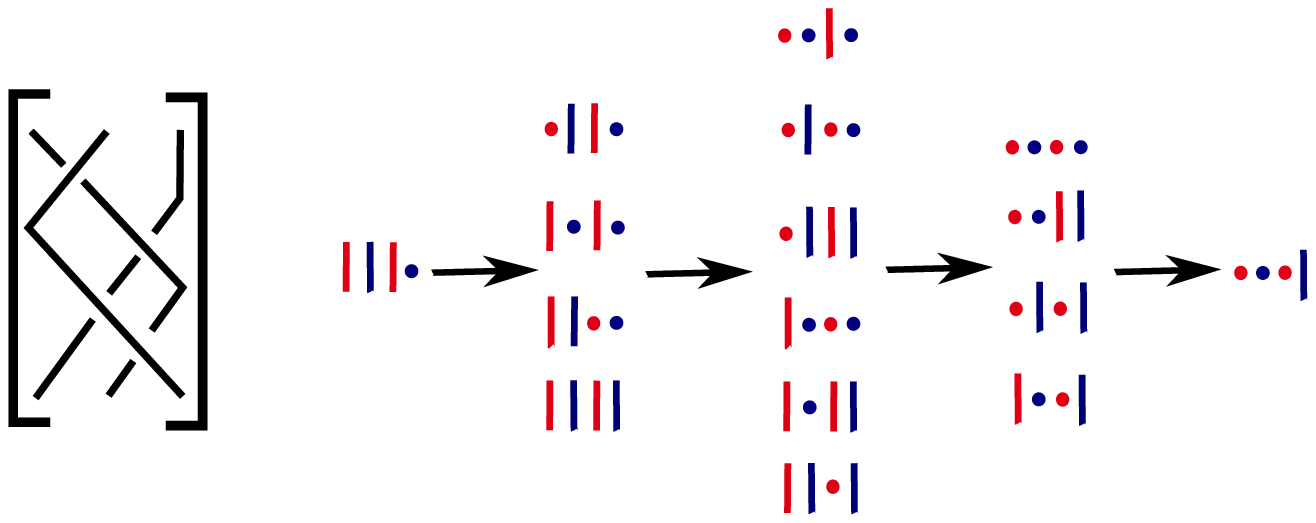}}
\label{example3}
\end{figure}\

\end{enumerate}

\section{Checking the Reidemeister moves}
\label{sec-rmoves}

We will use the matrix factorization construction of section \ref{subsec-mfconstruction} to check Reidemeister move I, as it is not very difficult to verify even over $\Z$ that this goes through, and the diagrammatic calculus of section \ref{subsec-diagrouquier} for the remaining moves. There are two main reasons for the interplay: first, checking Reidemeister II and III over $\Z$ using the matix factorization approach is rather computationally intensive (it was already quite so over $\Q$ in \cite{KR} with all the algebraic advantages of working over a field at hand); second, at this moment there does not exist a full diagrammatic description of Hochschild homology of Soergel bimodules, which prevents us from using a pictorial calculus to compute link homology from closed braid diagrams.  Of course, for Reidemeister II and III we could have used the computations of \cite{EK}, where we prove the stronger result that Rouquier complexes are functorial over braid cobordisms, but the proofs we exhibit below use essentially the same strategy as the original paper \cite{KR}, but are so much simpler and more concise that they underline well the usefulness of the diagrammatic calculus for computations. With that said, we digress...

A small lemma from linear algebra, which Bar-Natan refers to as ``Gaussian Elimination for Complexes" in \cite{BN2}, will be very helpfup to us.

\begin{lemma}
\label{lemma-GE}
  If $\phi:B \rightarrow D$ is an isomorphism (in some additive category $\cal C$), then the four term complex 
segment below

\begin{equation}
  \xymatrix@C=2cm{
    \cdots\
    \left[A\right]
    \ar[r]^{\begin{pmatrix}\alpha \\ \beta\end{pmatrix}} &
    {\begin{bmatrix}B \\ C\end{bmatrix}}
    \ar[r]^{\begin{pmatrix}
      \phi & \delta \\ \gamma & \epsilon
    \end{pmatrix}} &
    {\begin{bmatrix}D \\ E\end{bmatrix}}
    \ar[r]^{\begin{pmatrix} \mu & \nu \end{pmatrix}} &
    \left[F\right] \  \cdots
  }
\end{equation}
is isomorphic to the (direct sum) complex segment
\begin{equation}
  \xymatrix@C=3cm{
    \cdots\
    \left[A\right]
    \ar[r]^{\begin{pmatrix}0 \\ \beta\end{pmatrix}} &
    {\begin{bmatrix}B \\ C\end{bmatrix}}
    \ar[r]^{\begin{pmatrix}
      \phi & 0 \\ 0 & \epsilon-\gamma\phi^{-1}\delta
    \end{pmatrix}} &
    {\begin{bmatrix}D \\ E\end{bmatrix}}
    \ar[r]^{\begin{pmatrix} 0 & \nu \end{pmatrix}} &
    \left[F\right] \  \cdots
  }.
\end{equation}
Both of these complexes are homotopy equivalent to the (simpler)
complex segment
\begin{equation}
  \xymatrix@C=3cm{
    \cdots\
    \left[A \right]
    \ar[r]^{\left(\beta\right)} &
    {\left[C \right]}
    \ar[r]^{\left(\epsilon-\gamma\phi^{-1}\delta\right)} &
    {\left[E\right]}
    \ar[r]^{\left(\nu\right)} &
    \left[F\right] \  \cdots
  }.
\end{equation}
Here the capital letters are arbitrary columns of objects in $\cal C$ and all
Greek letters are arbitrary matrices representing morphisms with the
appropriate dimensions, domains and ranges (all the matrices are
block matrices); $\phi : B \rightarrow D$ is an isomorphism, i.e. it
is invertible.
\end{lemma}

\emph{Proof:} The matrices in complexes $(1)$ and $(2)$ differ by a
change of bases, and hence the complexes are isomorphic. $(2)$ and
$(3)$ differ by the removal of a contractible summand; hence, they
are homotopy equivalent. $\square$

\begin{figure} [!htbp]
\centerline{
\includegraphics[scale=.7]{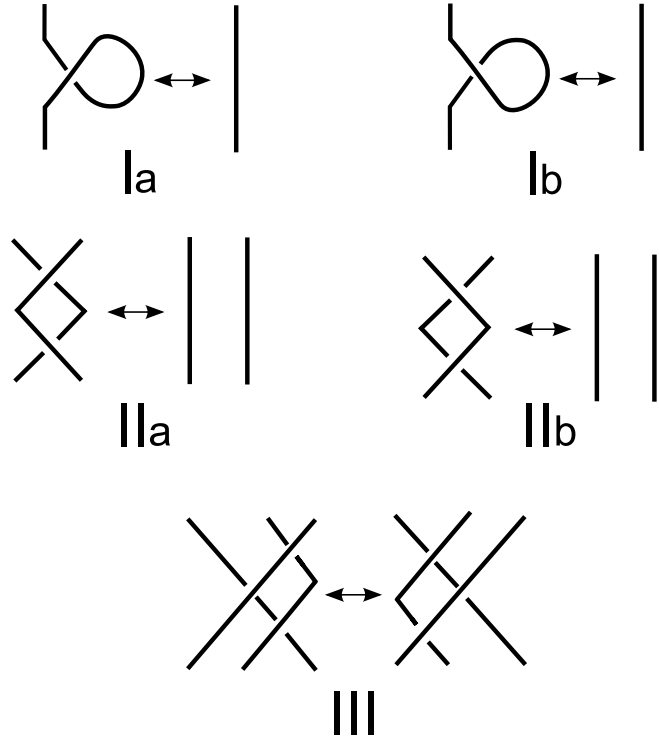}}
\caption{The Reidemeister moves}
\label{Rmoves}
\end{figure}\

\subsection{Reidemeister I}
\label{subsec-R1}

\begin{proof}
The complex $C(D_{I_a})$ for the left-hand side braid in Reidemester Ia, see figure \ref{Rmoves}, has the form 

\begin{displaymath}
\xymatrix{
\Z[x_1, x_2] \{0, -2, 0\}  \ar[rrr]^{0}    & & & \Z[x_1, x_2] \{0, 0, 0\}   \\
\Z[x_1, x_2] \{2, -2, -2\}   \ar[rrr]^{0}  \ar[u]^{(x_2 - x_1)}  & & & \Z[x_1, x_2] \{0, 0, -2\} \ar[u]^{1} }
\end{displaymath}

Up to homotopy, the right-hand side of the complex dissapears and only the top left corner survives after quotioning out by the relation $x_2 - x_1$. Note that the overall degree shifts of the total complex make sure that the left-over entry sits in the correct tri-grading.

Similarly, the complex $C(D_{I_b})$ for the left-hand side braid in Reidemester Ib, has the form

\begin{displaymath}
\xymatrix{
\Z[x_1, x_2] \{0, -2, 2\}  \ar[rrr]^{0}    & & & \Z[x_1, x_2] \{-2, 0, 2\}   \\
\Z[x_1, x_2] \{0, -2, 0\}   \ar[rrr]^{0}  \ar[u]^{1}  & & & \Z[x_1, x_2] \{0, 0, 0\} \ar[u]^{(x_2 - x_1)} }
\end{displaymath}

The left-hand side is annihilated and the upper-right corner remains modulo the relation $x_2 - x_1$. 
\end{proof}

\subsection{Reidemeister II}
\label{subsec-R2}
\begin{proof}
Lets first consider the braid diagrams for Reidemeister type IIa in figure \ref{Rmoves}. Recall the decomposition $B_i \otimes B_i \cong B_i\{-1\} \oplus B_i\{1\}$ in $\mc{SC}_2$  and its pictorial counterpart \ref{decomp1d}. The complex we are interested in is

\begin{figure} [!htbp]
\centerline{
\includegraphics[scale=.7]{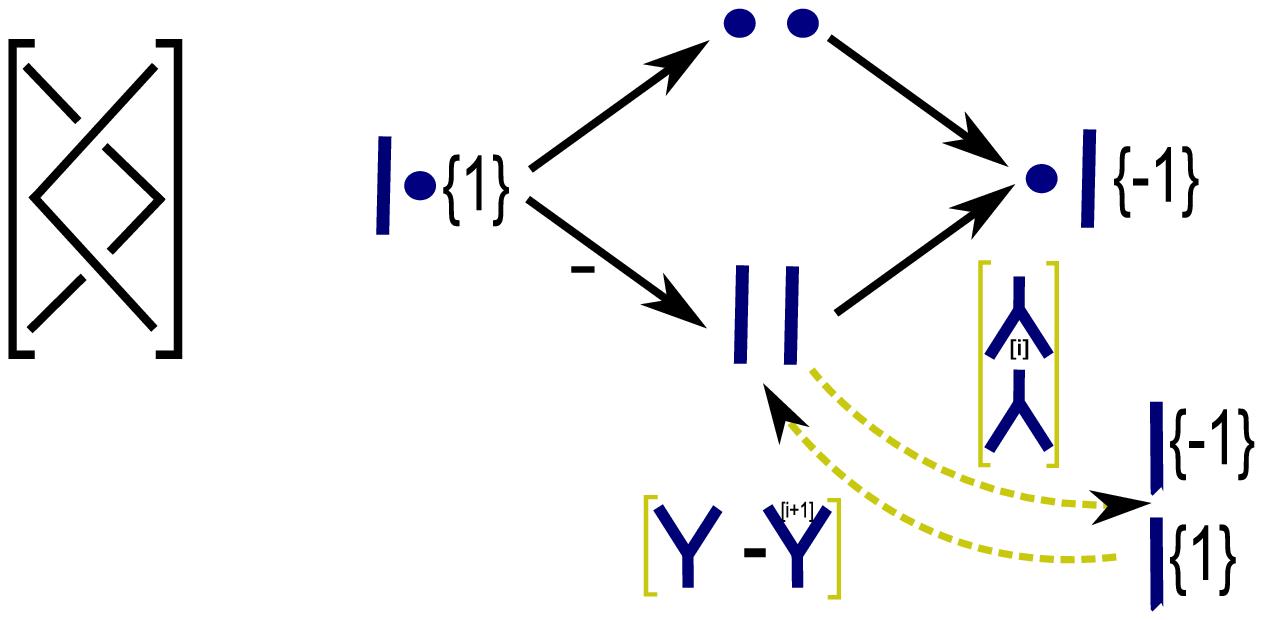}}
\caption{Reidemeister IIa complex with decomposition \ref{decomp1d}}
\label{R2Aa}
\end{figure}\

Inserting the decomposed $B_i \otimes B_i$ and the corresponding maps, we find two isomorphisms staring at us; we pick the left most one and mark it for removal.

\begin{figure} [!htbp]
\centerline{
\includegraphics[scale=.7]{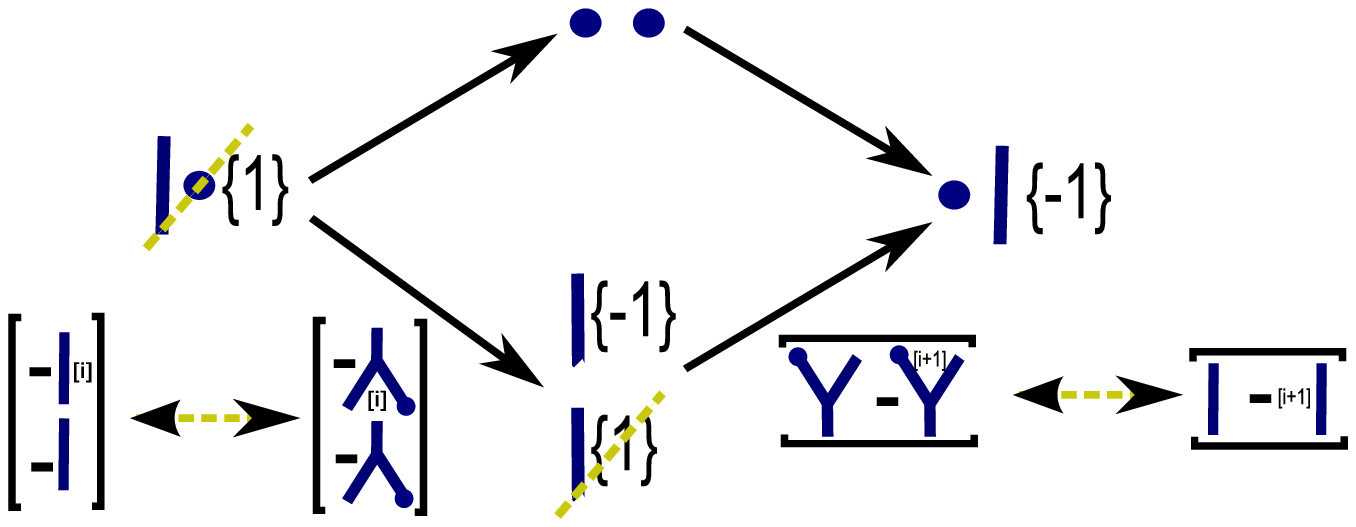}}
\caption{Reidemeister IIa complex, removing one of the acyclic subcomplexes}
\label{R2Ab}
\end{figure}\

After changing basis and removing the acyclic complex, as in Lemma \ref{lemma-GE}, we arrive at the complex below with two more entries marked for removal. 

\begin{figure} [!htbp]
\centerline{
\includegraphics[scale=.7]{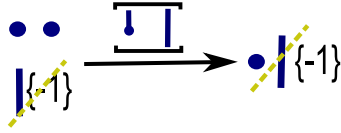}}
\caption{Reidemeister IIa complex, removing a second acyclic subcomplex}
\label{R2Ac}
\end{figure}\

With the marked acyclic subcomplex removed, we arrive at our desired result, the complex assigned to the no crossing braid of two strands as in figure \ref{Rmoves}.  The computation for Reidemeister IIb is virtually identical. \end{proof}

\subsection{Reidemeister III}
\label{subsec-R3}

\begin{proof}
Luckily, we only have to check one version of Reidemeister move III, but as the reader will see below even that is pretty easy and not much harder than that of Reidemeister II above. We follow closely the structure of the proof in \cite{KR}, utilizing the bimodule $R \otimes_{R^{i,i+1}} R \{-3\}$ and decomposition \ref{decomp2d} to reduce the complex for one of the RIII braids to that which is invariant under the move or, equivalently in our case, invariant under color flip. We start with the braid on the left-hand side of III in figure \ref{Rmoves};  the corresponding complex, with decomposition \ref{decomp1d} and \ref{decomp2d} given by dashed/yellow arrows, is

\begin{figure} [!htbp]
\centerline{
\includegraphics[scale=.7]{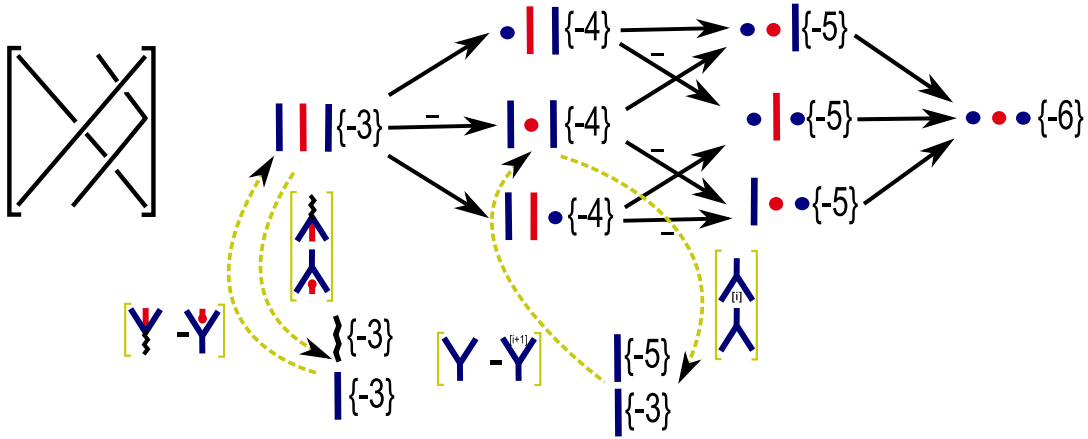}}
\caption{Reidemeister III complex with decompositions \ref{decomp1d} and \ref{decomp2d}}
\label{R3a}
\end{figure}\

We insert the decomposed bimodules and the appropriate maps; then we change bases as in Lemma \ref{lemma-GE} (the higher matrix of the two is before base-change, and the lower is after).

\begin{figure} [!htbp]
\centerline{
\includegraphics[scale=.8]{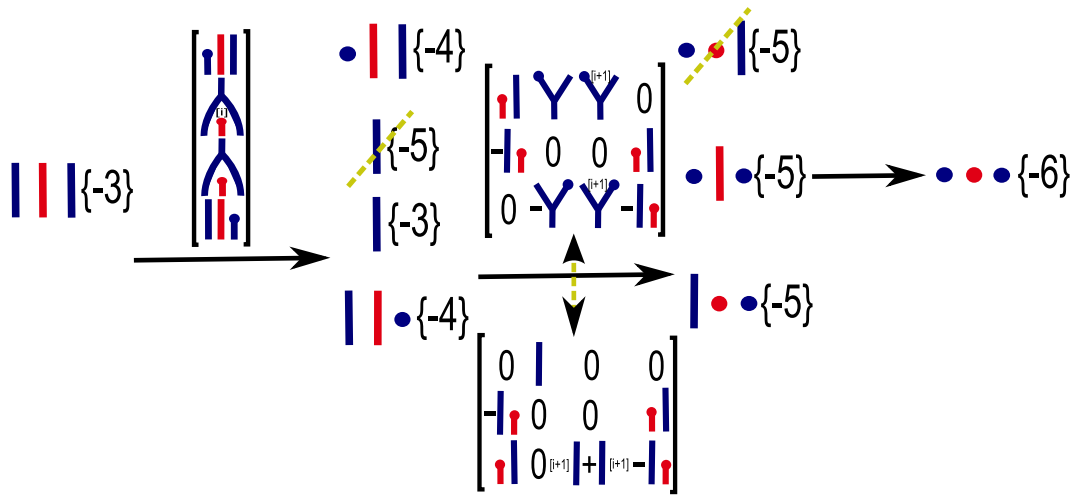}}
\caption{Reidemeister III complex, with an acyclic subcomplex marked for removal}
\label{R3b}
\end{figure}\

\pagebreak

We strike out the acyclic subcomplex and mark another one for removal; yet again we change bases (the lower matrix is the one after base change).

\begin{figure} [!htbp]
\centerline{
\includegraphics[scale=.8]{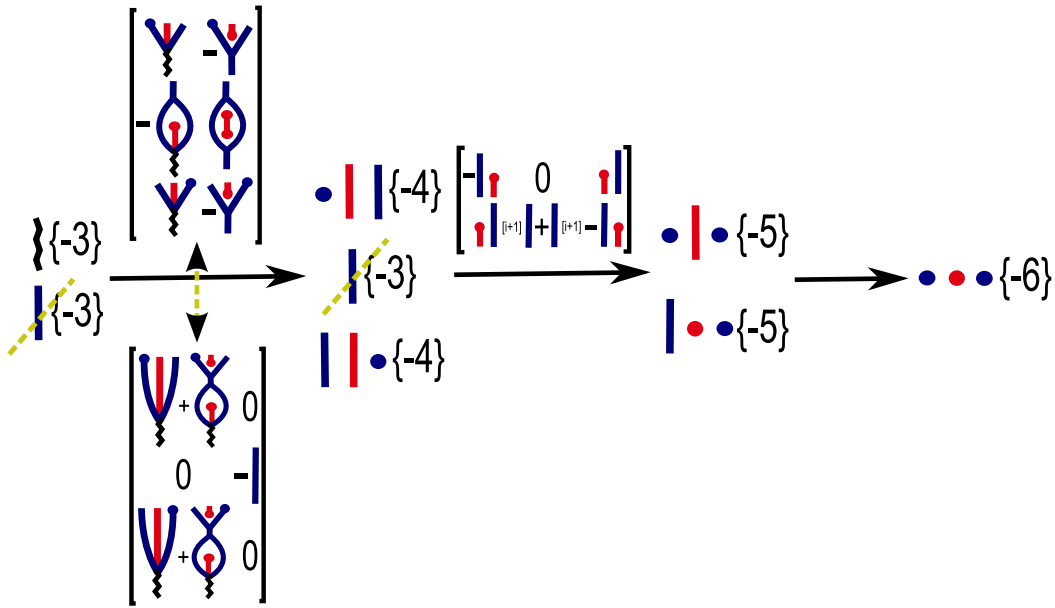}}
\caption{Reidemeister III complex, with another acyclic subcomplex marked for removal}
\label{R3c}
\end{figure}\

Now we are almost done; if we can prove that the maps

\begin{figure} [!htbp]
\centerline{
\includegraphics[scale=.8]{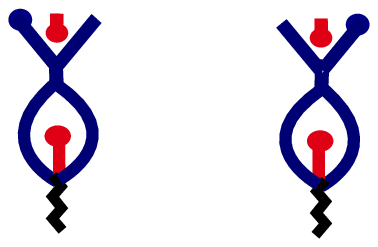}}
\label{R3e}
\end{figure}\

\hspace{-8mm} are invariant under color change, we would arrive at a complex that is invariant under Reidemeister move III. To do this we must stop for a second, go back to the source and examine the original, algebraic, definitions of the morphisms in \cite{EKh}; upon doing so we are relieved to see that the maps we are interested in are actually equal to zero (they are defined by sending $1 \otimes 1 \longmapsto 1 \otimes 1 \otimes 1 \otimes 1 \longmapsto 1 \otimes 1 \otimes 1 \longmapsto 0$). In all, we have arrived at

\begin{figure} [!htbp]
\centerline{
\includegraphics[scale=.7]{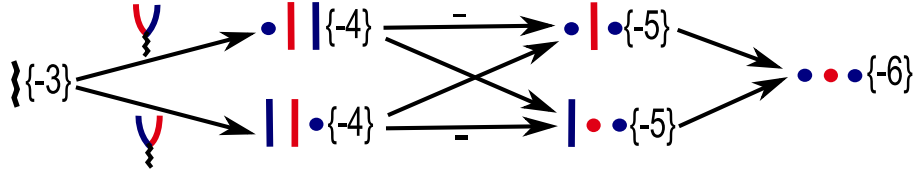}}
\caption{Reidemeister III complex - the end result, after removal of all acyclic subcomplexes}
\label{R3a}
\end{figure}\

Repeating the calculation for the braid on the right-hand side of RIII, figure \ref{Rmoves}, amounts to the above calculation with the  colors switched - a quick glance will convince the reader that the end result is the same complex rotated about the $x$-axis.
\end{proof}

\subsection{Observations}
\label{subsec-obs}

Having seen this interplay between the different constructions, perhaps it is a good moment to highlight exactly what categories we do need to work in so as to arrive at a genuine link invariant, or a braid invariant at that. Well, let us start with the latter: we can take the category of complexes of Soergel bimodules $\mc{KOM}(\mc{SC})$ (either the diagrammatic or ``original" version) and construct Rouquier complexes; if we mod out by homotopies and work  in $\mc{KOM}_h(\mc{SC})$, we arrive at something that is not only an invariant of braids but of braid cobordisms as well (over $\Z$ or $\Q$ if we wish). Now if we repeat the construction in the category of complexes of graded matrix factorizations $\mc{KOM}(\mathit{mf})$, we have some choices of homotopies to quotient out by. First, we can quotient out by the homotopies in the category of graded matrix factorizations and work in $\mc{KOM}(\mathit{hmf})$ and second, we can quotient in the category of the complexes and work in  $\mc{KOM}_h(\mathit{mf})$, or we can do both and work in $\mc{KOM}_h(\mathit{hmf})$. It is immediate that working in $\mc{KOM}_h(\mathit{mf})$ is necessary, but one could hope that it is also sufficient. A close look at the argument of  Claim \ref{claim-equivalence}, where the two constructions are proven equivalent, shows that if we start with the Koszul complex  associated to the resolution of a braid $D_{res}$  the polynomial relations coming from the singular vertices in $D_{res}$ form a regular sequence and, hence, the homology of this complex is the quotient of the edge ring $R(D_{res})$ by these relations and is supported in the right-most degree. It is this quotient that is isomorphic to the corresponding Soergel bimodule, i.e. the Koszul complex is quasi-isomorphic, as a bimodule, to $B'(D_{res})$. Hence, we really do need to work in $\mc{KOM}_h(\mathit{hmf})$, to have a braid invariant or an invariant of braid cobordisms, or a link invariant. 

Anyone, who has suffered throught the proofs of, say, Reidemeister III in \cite{KR} would probably find the above a relief. Of course, much of the ease in computation using this pictorial language is founded upon the intimate understanding and knowledge of hom spaces between objects in $\mc{SC}$, which is something that is only available to us due to the labors Elias and Khovanov in \cite{EKh}. With that said, it would not be suprising if this diagrammatic calculus would aid other calclulations of link homology in the future.

All in all we have arrived at an integral version of HOMFLY-PT link homology; combining with the results of \cite{EK} we have the following:

\begin{theorem}

Given a link $L \subset S^3$, the groups $H(L)$ and $\olH(L)$ are invariants of $L$ and when tensored with $\Q$ are isomorphic to the unreduced and reduced versions, respectively, of the Khovanov-Rozansky HOMFLY-PT link homology. Moreover, these integral homology theories give rise to functors from the category of braid cobordisms to the category of complexes of graded $R$-bimodules.
\end{theorem}

\section{Rasmussen's spectral sequence and integral $sl(n)$-link homology}
\label{sec-ss}

It is time for us to look more closely at Rasmussen's spectral sequence from HOMFLY-PT to $sl(n)$-link homology. For this we need an extra ``horizontal" differential $d_-$ in our complex, and here is the first time we encounter matrix factorizations with a non-zero potential; as before, to a link diagram $D$ we will associate the tensor product of complexes of matrix factorizations with potential for each crossing. These will be complexes over the ring 
$$R_c = \Z[x_i, x_j, x_k, x_l]/(x_k + x_l - x_i - x_j) \cong \Z[x_i, x_j, x_k],$$

\hspace{-6.5mm}with total potential

$$W_p[x_i, x_j, x_k, x_l] = p(x_k) + p(x_l) - p(x_i) - p(x_j),$$

\hspace{-6.5mm}where the $p(x) \in \Z[x]$. We do not specify the potential $p(x)$ at the moment as the spectral sequence works for any choice; later on when looking at $sl(n)$-link homology we will set $p(x) = x^{n+1}$.

To define $d_-$, let $p_i = W_p/(x_k - x_i)$ and $p_{ij} = -W_p/(x_k - x_i)(x_k - x_j)$ (recall that in $R_c$,  $(x_k - x_i)(x_k - x_j) = x_i x_j - x_k x_l$, and note that these polynomials are actually in $R_c$).

To the positive crossing $D_+$ we assign the following complex:

\begin{displaymath}
\xymatrix{
R_c \{0, -2, 0\}  \ar@<.5ex>[rrr]^{(x_k - x_i)}    & & & R_c \{0, 0, 0\} \ar@<.5ex>[lll]^{p_i}  \\
R_c \{2, -2, -2\}   \ar@<.5ex>[rrr]^{-(x_k - x_i)(x_k - x_j)}  \ar[u]^{(x_j - x_k)}  & & & R_c \{0, 0, -2\} \ar@<.5ex>[lll]^{p_{ij}} \ar[u]^{1} }
\end{displaymath}

To the negative crossing $D_-$ we assign the following complex:

\begin{displaymath}
\xymatrix{
R_c \{0, -2, 2\}  \ar@<.5ex>[rrr]^{-(x_k - x_i)(x_k - x_j)}    & & & R_c \ar@<.5ex>[lll]^{p_{ij}} \{-2, 0, 2\}   \\
R_c \{0, -2, 0\}   \ar@<.5ex>[rrr]^{(x_k - x_i)}  \ar[u]^{1}  & & & R_c \{0, 0, 0\} \ar@<.5ex>[lll]^{p_i} \ar[u]^{(x_j - x_k)} }
\end{displaymath}

The total complex for a link $L$ with diagram $D$ will be defined analagously to the one above, i.e. 

$$ C_p(D) := \bigotimes_{crossings} (C(D_c) \otimes_{R_c} R(D)),$$

\hspace{-6.5mm}as will be the reduced $\olH_p(L,i)$ and unreduced $H_p(L)$ versions of link homology.

The main result of \cite{Ras2} is the following:

\begin{theorem} \{Rasmussen, \cite{Ras2}\}
Suppose $L \subset S^3$ is a link, and let $i$ be a marked component of $L$. For each $p(x) \in \Q[x]$, there is a spectral sequence $E_k(p)$ with $E_1(p) \cong \olH(L)$ and $E_\infty(p) \cong \olH_p(L,i)$. For all $k>0$, the isomorphism type of $E_k(p)$ is an invariant of the pair $(L,i)$.
\end{theorem}

In particular setting $p(x) = x^{n+1}$ one would arrive at a spectral sequence from the HOMFLY-PT to the $sl(n)$-link homology. Rasmussen's result pertains to rational link homology with matrix factorizations defined over the ring $\Q[x_1, \dots, x_n]$ and potentials polynomials in $\Q[x]$. We will essentially repeat his construction in our setting and, for the benefit of those familiar with the results of \cite{Ras2}, will stay as close as possible to the notation and conventions therein. This will be a rather condensed version of the story and we refer the reader to the original paper for more details.

We will work primarily with reduced link homology (although all the results follow through for both versions) and with closed link diagrams, where all three differentials $d_v$, $d_+$, and $d_-$ anticommute. We have some choices as to the order of running the differentials, so let us define

$$\olH^+(D, i) = H(\olC(D, i), d_+).$$

Here, $\olH^+(D, i)$ inherits a pair of anticommuting differentials $d^*_-$ and $d^*_v$, where $d^*_-$ lowers $gr_h$ by $1$ while preserving $gr_v$ and $d^*_v$ raises $gr_v$ by $1$ while preserving $gr_h$. Hence, $(\olH^+_p(D,i), d^*_v, d^*_-)$ defines a double complex with total differential $d_{v-} :=  d^*_v +  d^*_-$.  

\begin{defn}
\label{defn-ss}
Let $E_k(p)$ be the spectral sequence induced by the horizontal filtration on the complex $(\olH_p^+(D, i), d_{v-})$.
\end{defn}

After shifting the triple grading of $E_k(p)$ by $\{-w+b-1,w+b-1,w-b+1\}$ it is immediate that the first page of the spectral sequence is isomorphic to $\olH(L,i)$ (the part of the differential $d^*_v +  d^*_-$ which preserves horizontal grading on $E_0(p) = \olH^+(D, i)\{-w+b-1,w+b-1,w-b+1\}$ is precisely $d^*_v$, i.e. $d_0(p) = d^*_v$ and  $$E_1(p) = H(\olH^+(D, i), d^*_v)\{-w+b-1,w+b-1,w-b+1\} \cong \olH(L,i),$$ where $D$ is a diagram for $L$). It also follows that $d_k(p)$ is homogenous of degree $-k$ with respect to  $gr_h$ and degree $1-k$ with respect to $gr_v$, and in the case that $p(x) = x^{n+1}$ it is also homogeneous of degree $2nk$ with respect to the $q$-grading.

\begin{claim}
Suppose $L \subset S^3$ is a link, and let $i$ be a marked component of $L$. For each $p(x) \in \Z[x]$, the spectral sequence $E_k(p)$ has $E_1(p) \cong \olH(L,i)$ and $E_\infty(p) \cong \olH_p(L,i)$. For all $k>0$, the isomorphism type of $E_k(p)$ is an invariant of the pair $(L,i)$. 
\end{claim}

\begin{proof} 
We argue as in \cite{Ras2} section 5.4. Suppose that we have two closed diagrams $D_j$ and $D_j'$ that are related by the $j$'th Reidemeister move, and suppose that there is a morphism 
$$\sigma_j:\olH^+_p(D_j, i) \rightarrow \olH^+_p(D_j',i)$$
in the category  $\mc{KOM}(\mathit{mf})$ that extends to a homotopy equivalence in the category of modules over the edge ring $R$. Then $\sigma_j$ induces a morphism of spectral sequences $(\sigma_j)_k: E_k(D_j, i, p) \rightarrow E_k(D_j', i, p)$ which is an isomorphism for $k>0$. See \cite{Ras2} for more details and discussion. Hence, in practice we have to exhibit morphisms and prove invariance for the first page of the spectral sequence, i.e. for the HOMLFY-PT homology, which is basically already done. However, we ought to be a bit careful, of course, as here we are working with $\olH^+_p(D,i)$ and not with the complex $\olC(D,i)$ defined in section \ref{sec-rmoves}.

Reidemeister I is done, as in this case $d_+=0$ and, hence, the complex $\olH^+_p(D,i) = \olC_p(D,i)$ and the same argument as the one in section \ref{subsec-R1} works here.

For Reidemesiter II and III, we have to observe that for a closed diagram we have morphisms $\sigma_j: \olC_p(D_j,i) \rightarrow \olC_p(D_j',i)$ for $j = II, III$, which are homotopy equivalences of complexes (these can be extrapolated from section \ref{sec-rmoves} above, or from \cite{EK}, where all chain maps are exhibited concretely). Therefore we get induced maps $(\sigma_j)_k$ on the spectral sequence with the property that $(\sigma_j)_1 = \sigma_{j*}$ is an isomorphism. 

To get the last part of the claim, i.e. that the reduced homology depends only on the link component and not on the edge therein we refer the reader to \cite{Ras2}, as the arguments from there are valid verbatum. \end{proof}

Setting $p(x) = x^{n+1}$, we get that the differentials $d_k(p)$ preserve $q+2ngr_h$ and, hence, the graded Euler characteristic of $H(\olH^+_p(D,i), d_{v-})$ with respect to this quantity is the same as that of $E_1(x^{n+1})$. Tensoring with $\Q$, to get rid of torsion elements, and computing we see that the Euler characteristic of the $E_\infty(x^{n+1})$ is the quantum $sl(n)$-link polynomial $P_L(q^n, q)$ of $L$. See \cite{Ras2} section $5.1$ for details.  We have arrived at:

\begin{theorem}
The $E_\infty(x^{n+1})$ of the spectral sequence defined in \ref{defn-ss} is an invariant of $L$ and categorifies the quantum $sl(n)$-link polynomial $P_L(q^n, q)$.
\end{theorem}

\begin{remark}
Well, we have a categorification over $\Z$ of the quantum $sl(n)$-link polynomial, but what homology theory exactly are we dealing with? Is it isomorphic to \\
 $H(H(H(\olC_{x^{n+1}}(D,i), d_+),d^*_-), d^*_v)$ or to  $H(H(\olC_{x^{n+1}}(D,i), d_+ + d_-), d^*_v)$ and are these two isomorphic here? The answer is not immediate. In \cite{Ras2}, Rasmussen bases the corresponding results on a lemma that utilizes the Kunneth formula, which is much more manageable in this context when looked at over $\Q$. Of course, for certain classes of knots things are easier. For example, if we take the class of knots that are \emph{KR-thin}, then the spectral sequence converges at the $E_1$ term, as this statemtent only depends on the degrees of the differentials, and we have that  $E_\infty(x^{n+1}) \cong H(H(\olC_{x^{n+1}}(D,i), d_+), d^*_v)$. However, that's a bit of a `red herring' as stated.
\end{remark}


\vspace{10mm}

\hspace{-5mm}\emph{Daniel Krasner, Department of Mathematics, Columbia University, New York, NY 10027} \\
\vspace{3mm}
E-mail: dkrasner@math.columbia.edu\\


\begin{thebibliography}{1}

\bibitem{BN2} D. ~Bar-Natan, Fast Khovanov homology computations, \emph{Journal of Knot Theory and Its Ramifications}, 16-3 (2007) 243-255. .

\bibitem{CS} J.~Carter and M.~Saito, Knotted surfaces and their diagrams, \emph{Math.
Surv. and Mon.} 55, AMS, 1998.

\bibitem{CMW} D.~Clark, S.~Morrison and K.~Walker, Fixing the functoriality of Khovanov homology, 	arXiv:math/0701339v2.

\bibitem{EKh} B.~Elias and M.~Khovanov, Diagrammatics for Soergel Categories, 2009, math.QA/0902.4700v1.

\bibitem{EK} B.~Elias and D.~Krasner, Rouquier complexes are functorial over braid cobordisms, 2009, arXiv:0906.4761v2.

\bibitem{Ka} C. ~Kassel, Homology and cohomology of associative algebras: a concise introduction to cyclic homology. Online notes. 

\bibitem{Kau} L. H. ~Kauffman, Formal knot theory, Number 30 in \emph{Mathematical Notes}, Princeton University
Press, 1983.

\bibitem{Kh1} M. Khovanov, A categorification of the Jones polynomial, \emph{Duke Math J.} 101, 3, 359-426, 1999, arxiv math.QA/9908171.

\bibitem{K1} M. ~Khovanov, Triply-graded link homology and Hochschild homology of Soergel bimodules,  \emph{Internat. J. Math.}  18  (2007),  no. 8, 869--885.

\bibitem{KR1} M. Khovanov and L. Rozansky, Matrix factorizations and link homology, math.QA/0401268,  \emph{Fundamenta Mathematicae} Vol. 199, num. 1, 2008.

\bibitem{KR} M. ~Khovanov and L. ~Rozansky, Matrix factorizations and link homology II, \emph{Geom. Topol.}  12  (2008),  no. 3, 1387--1425.

\bibitem{KT} M. ~Khovanov and R. ~Thomas, Braid cobordisms, triangulated categories, and flag varieties  \emph{Homology, Homotopy Appl.}  9  (2007),  no. 2, 19--94.

\bibitem{KL} M. ~Khovanov, A. ~Lauda, A diagrammatic approach to categorification of quantum groups I, \emph{Represent. Theory} 13 (2009), 309-347.

\bibitem{MOY} H. ~Murakami, T. ~Ohtsuki and S. ~Yamada, HOMFLY polynomial via an invariant of colored plane graphs, \emph{Enseign. Math} (2) 44 (1998), 325-360.  

\bibitem{MV} M. Mackaay and P. Vaz, The reduced HOMFLY-PT homology for the Conway and the Kinoshita-Terasaka knots,    arXiv:0812.1957v1.

\bibitem{Ras1} J. ~Rasmussen, Khovanov-Rozansky homology of two-bridge knots and links,  \emph{Duke Math. J.}  136  (2007),  no. 3, 551--583.

\bibitem{Ras2} J.~Rasmussen, Some differentials on Khovanov-Rozansky homology, arXiv:math/0607544v2.

\bibitem{Rou1} R.~Rouquier, Categorification of the braid groups, math.RT/0409593.

\bibitem{Soe1}  W.~Soergel, The combinatorics of Harish-Chandra bimodules,
	\emph{Journal Reine Angew. Math.} {\bf 429}, (1992) 49--74.

\bibitem{Soe2}  W.~Soergel, Gradings on representation categories,
	\emph{Proceedings of the ICM 1994 in Z\"urich}, 800--806, Birkh\"auser, Boston.

\bibitem{Soe3} W.~Soergel, Combinatorics of Harish-Chandra modules,
	\emph{Proceedings of the NATO ASI 1997, Montreal, on Representation theories and
	Algebraic geometry}, edited by A.~Broer, Kluwer (1998).

\bibitem{Soe4} W.~Soergel, Kazhdan-Lusztig-Polynome und unzerlegbare
	Bimoduln "uber Polynomringen,  math.RT/0403496v2, english translation available
	on the author's webpage.

\bibitem{W1} B. ~Webster, Khovanov-Rozansky homology via a canopolis formalism,  \emph{Algebr. Geom. Topol.}  7  (2007), 673--699.

\bibitem{W2} B. ~Webster, Kr.m2. http://katlas.math.toronto.edu/wiki/user:Ben/KRhomology, 2005.

\end{thebibliography}
\end{document}